\documentclass[11pt]{amsart}

\usepackage{amsfonts,amssymb,amsthm,eucal,color,bbm,verbatim,enumitem,graphicx,bm}

\usepackage[french,english]{babel}
\usepackage[margin=1.25in]{geometry}

\usepackage{hyperref}

\newtheorem{thm}{Theorem}
\newtheorem{cor}[thm]{Corollary}
\newtheorem{lemma}[thm]{Lemma}
\newtheorem{prop}[thm]{Proposition}

\newcommand{\hess}{\mathrm{Hess\,}}
\newcommand{\R}{\mathbb{R}}
\newcommand{\E}{\mathbb{E}}

\newcommand{\Prob}{\mathbb{P}}
\newcommand{\N}{\mathbb{N}}

\newcommand{\C}{\mathbb{C}}

\newcommand{\I}{\textbf{1}}

\DeclareMathOperator{\tr}{tr}

\newcommand{\inprod}[2]{\left\langle #1, #2 \right\rangle}
\renewcommand{\Re}{\operatorname{Re}}

\newcommand{\abs}[1]{\left\vert #1 \right\vert}
\newcommand{\norm}[1]{\left\Vert #1 \right\Vert}

\newcommand{\eps}{\varepsilon}

\DeclareMathOperator{\var}{Var}
\DeclareMathOperator{\cov}{Cov}

\newcommand{\Set}[2]{\left\{#1 \mathrel{} \middle| \mathrel{} #2 \right\}}

\newcommand{\Orthogonal}[1]{\mathbb{O}\left(#1\right)}

\newcommand{\mat}[2]{\mathcal{M}_{#1}(#2)}
\newcommand{\symmat}[2]{\mathcal{M}^s_{#1}(#2)}
\newcommand{\asymmat}[2]{\mathcal{M}^{as}_{#1}(#2)}

\allowdisplaybreaks[3]

\author{Elizabeth S.\ Meckes and Mark W.\ Meckes}

\address{Department of Mathematics, Applied Mathematics, and
  Statistics, Case Western Reserve University, 10900 Euclid Ave.,
  Cleveland, Ohio 44106, U.S.A.}

\email{elizabeth.meckes@case.edu}

\address{Department of Mathematics, Applied Mathematics, and
  Statistics, Case Western Reserve University, 10900 Euclid Ave.,
  Cleveland, Ohio 44106, U.S.A.}

\email{mark.meckes@case.edu}

\title[Spectrum of rotationally invariant random matrix
ensembles]{Fluctuations of the spectrum in rotationally invariant random matrix ensembles}

\begin{document}

\begin{abstract}
  We investigate traces of powers of random matrices whose
  distributions are invariant under rotations (with respect to the
  Hilbert--Schmidt inner product) within a real-linear subspace of the
  space of $n\times n$ matrices.  The matrices we consider may be real or
  complex, and Hermitian, antihermitian, or general.  We use Stein's
  method to prove multivariate central limit theorems, with
  convergence rates, for these traces of powers, which imply central
  limit theorems for polynomial linear eigenvalue statistics.  In
  contrast to the usual situation in random matrix theory, in our
  approach general, nonnormal matrices turn out to be easier to study
  than Hermitian matrices.
\end{abstract}

\maketitle

\section{Introduction}

The limiting behavior of the eigenvalues of random matrices is a
central problem in modern probability, with applications and
connections in statistics, physics, and beyond.  The eigenvalues of
the classical ensembles have been studied extensively, and much is
known.  However, there are many other ensembles which are natural in
applied contexts that have been less thoroughly explored.  In this
paper, we study the eigenvalues of rotationally invariant random
matrix ensembles; i.e., probability measures on real-linear spaces of
$n \times n$ matrices which are invariant under rotations within those
spaces.  We emphasize that this is different from the more common
assumption of invariance under conjugation by orthogonal or unitary
$n\times n$ matrices; ensembles with the latter invariance property
are most often referred to as ``matrix models'' or as ``orthogonally
invariant'' or ``unitarily invariant'' ensembles, respectively, but
are unfortunately also sometimes referred to as rotationally
invariant.  The spaces we consider include the spaces of all real or
complex $n \times n$ matrices, the space of all $n \times n$ Hermitian
matrices, or others.  The classical Gaussian random matrix ensembles
are of this type, and so are random matrices chosen uniformly from the
sphere with respect to the Hilbert--Schmidt norm.  Beyond the
classical Gaussian cases, such ensembles have been studied in the
physics literature (see, e.g.,
\cite{AdToKu,CaDe,DeCa,Guhr,LiZh,MuKl}), frequently under the names
``fixed trace ensembles'' (for matrices uniformly distributed on a
sphere for the Hilbert--Schmidt norm) or ``norm-dependent ensembles'';
fixed trace ensembles have also been investigated in the numerical
analysis literature \cite{Demmel,Edelman-cond,Edelman-det} and in more
mathematically oriented work on random matrix theory
\cite{EiKn,GoGoLe,GoGo}.

In this paper we investigate the fluctuations of traces of powers of
such random matrices, showing that these fluctuations have a jointly
Gaussian distribution, under certain hypotheses, in the
high-dimensional limit.  This implies, in particular, that linear
eigenvalue statistics $\sum_{j=1}^n f(\lambda_j)$ are asymptotically
Gaussian, where $\lambda_1, \dots, \lambda_n$ are the eigenvalues of
our random matrix and $f$ is a polynomial function.  Gaussian limits
for fluctuations of linear eigenvalue statistics have been studied
intensively for other random matrix ensembles; we mention in
particular \cite{AnZe,BaSi,Johansson2,LyPa,Shcherbina,SiSo} for
Wigner-type matrices (random Hermitian matrices whose entries on and
above the diagonal are independent),
\cite{DiEv,DoSt,DoSt2,Fulman2,Johansson,Soshnikov,Stein,Wieand} for
Haar-distributed random matrices from the classical compact groups,
and \cite{CiErSc,NoPe,Rider,RiSi,RiVi} for the typically most
difficult case of random matrices with all independent entries.

Our proofs are based on the infinitesimal or continuous version
Stein's method of exchangeable pairs, which has found a number of
applications in random matrix theory, and which is particularly well
suited to the analysis of settings like ours that exhibit continuous
geometric symmetries.  This method has been used to prove central
limit theorems for linear eigenvalue statistics for various random
matrix ensembles in \cite{DoSt,DoSt2,Fulman2,JoSm,JoSmWe,Stein,Stolz};
other applications in random matrix theory appear in
\cite{ChMe,Fulman1,EM-linear,MM-quantum}.  We also mention
\cite{Chatterjee-gauge}, which does not apply Stein's method for
distributional approximation but uses a continuous family of
exchangeable pairs to prove identities for expectations, similar to
our proof of Theorem \ref{T:traces-of-powers-means} below; and
\cite{Chatterjee,NoPe}, which apply other versions of Stein's method
to investigate linear eigenvalue statistics of random matrices.

An unusual feature of our proofs is that they allow a unified approach
to both the Hermitian and non-Hermitian cases. More surprisingly, it
turns out that, in contrast to the usual situation in random matrix
theory, the non-Hermitian case is easier to handle here, for reasons
that will be discussed below.

One can reasonably object that in the non-Hermitian case it is natural
to consider more general linear eigenvalue statistics; for example in
the polynomial setting one should allow the test function $f$ to be a
polynomial in both $z$ and $\overline{z}$. However, as in the present
work, most known central limit theorems for linear eigenvalue
statistics for non-normal random matrices require $f$ to be analytic
or otherwise highly restricted (as in, e.g.,
\cite{NoPe,ORoRe,Rider,RiSi}), and even so, the proofs are more
difficult than in the Hermitian case.  An exception is the very recent
work \cite{CiErSc}, which handles random matrices with i.i.d.\ complex
entries and test functions with only $2+\epsilon$ derivatives.

We now turn to a more precise description of the random matrix
ensembles we consider and our results.

The random matrices we consider are drawn from a real-linear subspace
$V$ of the space $\mat{n}{\C}$ of $n\times n$ matrices over $\C$. We
take $V$ to be one of the following: $\mat{n}{\C}$ itself; the space
$\mat{n}{\R}$ of $n\times n$ matrices over $\R$; the space
$\symmat{n}{\R}$ of real symmetric $n\times n$ matrices; the space
$\symmat{n}{\C}$ of complex Hermitian $n\times n$ matrices; the space
$\asymmat{n}{\R}$ of real antisymmetric $n\times n$ matrices; and the
space $\asymmat{n}{\C}$ of complex anti-Hermitian $n\times n$
matrices.  All of these spaces are real inner product spaces with
respect to the inner product $\inprod{A}{B}=\Re\tr(AB^*)$, and have
the associated Hilbert--Schmidt norm $\norm{A} = \sqrt{\tr (AA^*)}$.
(We will also make some use of the complex (Hilbert--Schmidt) inner
product, and the operator norm $\norm{A}_{op}$.)

The distributions we consider on $V$ are rotationally invariant in the
sense that they are invariant under linear isometries of the entire
space $V$ equipped with this inner product; this is stronger than the
more commonly considered property of invariance under multiplication
or conjugation by a unitary matrix in $\mat{n}{\C}$. If $X \in V$ has
a rotationally invariant distribution, then we can write
$X = \norm{X} \widetilde{X}$, where $\widetilde{X}$ is uniformly
distributed on the unit sphere (with respect to the Hilbert--Schmidt
norm) of $V$ and is independent from $\norm{X}$. (In fact, in the
proofs below it will be convenient to use a slightly different
normalization for $\widetilde{X}$.)

Rotationally invariant distributions can also be described concretely
in terms of orthonormal bases on each space.  Let $E_{jk}$ denote the
$n\times n$ matrix with a one in the $(j,k)$ position and zeroes
everywhere else.  For $j<k$, let
$F_{jk}=\frac{1}{\sqrt{2}}(E_{jk}+E_{kj})$ and
$G_{jk}=\frac{1}{\sqrt{2}}(E_{jk}-E_{kj})$.  Let $d$ be the (real)
dimension of $V$. We denote by $\{B_\alpha\}_{\alpha=1}^d$ orthonormal
bases (with respect to the real inner product
$\inprod{A}{B}=\Re\tr(AB^*)$) for the spaces $V$ above, as follows.
\begin{center}\begin{tabular}{|c|c|}
\hline
$V$&$\phantom{\Big|}\{B_\alpha\}_{\alpha=1}^d$\\
\hline\hline
$\mat{n}{\C}$&$\phantom{\Big|}\{E_{jk}\}_{1\le j,k\le n}\cup \{iE_{jk}\}_{1\le
               j,k\le n}$\\
\hline
$\mat{n}{\R}$&$\phantom{\Big|}\{E_{jk}\}_{1\le j,k\le n}$\\
\hline
$\symmat{n}{\R}$&$\phantom{\Big|}\{E_{jj}\}_{1\le j\le n}\cup \{F_{jk}\}_{1\le
               j<k\le n}$\\
\hline
$\symmat{n}{\C}$&$\phantom{\Big|}\{E_{jj}\}_{1\le j\le n}\cup \{F_{jk}\}_{1\le
               j<k\le n}\cup \{iG_{jk}\}_{1\le
               j<k\le n}$\\
\hline
$\asymmat{n}{\R}$&$\phantom{\Big|} \{G_{jk}\}_{1\le
               j<k\le n}$\\
\hline
$\asymmat{n}{\C}$&$\phantom{\Big|}\{iE_{jj}\}_{1\le j\le n}\cup \{G_{jk}\}_{1\le
               j<k\le n}\cup \{iF_{jk}\}_{1\le
               j<k\le n}$\\
\hline
\end{tabular}\end{center}

For each choice of $V$, consider a random vector
$\{X_\alpha\}_{\alpha=1}^d$ with a rotationally invariant distribution
in $\R^d$, normalized so that \(\E \sum_{\alpha=1}^d X_\alpha^2=n,\) and define
$$X=\sum_{\alpha=1}^{d}X_\alpha B_\alpha.$$ 
The random matrix $X \in V$ then has a rotationally invariant
distribution in $V$ and satisfies \(\E\|X\|^2=n\). Note that choosing
the random vector $\{X_\alpha\}_{\alpha=1}^d$ according to a Gaussian
distribution results in various classical random matrix ensembles: in
the case of unrestricted real or complex matrices, we have the real,
respectively complex Ginibre ensembles, and in the case of real
symmetric or complex Hermitian matrices, we have the Gaussian
Orthogonal Ensemble (GOE) and Gaussian Unitary Ensemble (GUE),
respectively.

Our first main result identifies the means of the random variables
$W_p = \tr X^p$ for $p \in \N$.  This result is essentially known, and
can easily be deduced from the Gaussian cases, where the
classical proofs make essential use of the independence of the
entries.  Here we give an independent proof which is an easy
by-product of the analysis of the exchangeable pair used to prove
Theorems \ref{T:traces-of-powers-clt-nonnormal} and
\ref{T:traces-of-powers-clt-normal} below.  (As noted above, a similar
approach was used in \cite{Chatterjee-gauge} to prove identities for
expectations of functions of random orthogonal matrices.)

\begin{thm}\label{T:traces-of-powers-means}
  Let $X$ be a random matrix in $V\subseteq \mat{n}{\C}$ as above,
  whose distribution is invariant under rotations of V.

  Suppose that $\E\|X\|^2=n$ and that for each $k$, there is a
  constant $\alpha_k$ depending only on $k$ such that
  \[
    t_k(X)=\left|n^{-k/2}\E\|X\|^{k}-1\right|\le\frac{\alpha_k}{n}.
  \]

  For $p\in\N$, let $W_p=\tr(X^p)$.  In all cases, if $p$ is odd, then
  $\E W_p=0$.  For $p=2r$,
  \[
    \E W_p=\begin{cases}0 & \text{if } V=\mat{n}{\C},\\
      1+O\left(\frac{1}{n}\right) & \text{if } V=\mat{n}{\R},\\
      nC_r+O(1) & \text{if } V=\symmat{n}{\C} \text{ or } \symmat{n}{\R},\\
      (-1)^{r}nC_r+O(1) & \text{if }
      V=\asymmat{n}{\C} \text{ or } \asymmat{n}{\R},
    \end{cases}
  \]
  where $C_r = \frac{1}{r+1}\binom{2r}{r}$ is the $r$th Catalan
  number. 
\end{thm}

In Theorem \ref{T:traces-of-powers-means}, as well as all the
following results, the $O$ terms refer to $n \to \infty$, with implied
constants that may depend on $p$ (or $m$ below) and the constants
$\alpha_k$, but do not otherwise depend on the precise distribution of
$X$.

In just the first case of Theorem \ref{T:traces-of-powers-means} (when
$V = \mat{n}{\C}$), the hypothesis on $t_k(X)$ can be replaced by the
weaker assumption that $t_k(X) < \infty$ for each $k$; that is, simply
that all moments of $\norm{X}$ are finite.  In the other cases that
hypothesis can be weakened to assuming each $t_k(X)$ is $o(1)$, at the
expense of more complicated version of the error terms.  We have
chosen here to assume a simple and quite mild hypothesis that lets us
state a clean result.

Theorems \ref{T:traces-of-powers-clt-nonnormal} and
\ref{T:traces-of-powers-clt-normal} describe the fluctuations of the
$W_p$, formulated as  comparisons of integrals of $C^2$ test
functions.  In what follows, for $g\in C^2(\R^m)$,
\[
M_1(g)=\sup_{x\in\R^m}|\nabla g(x)|
\] denotes the Lipschitz constant of $g$ and 
\[
M_2(g)=\sup_{x\in\R^m}\|\hess(g)(x)\|_{op}
\]  
the maximum operator norm of the Hessian of $g$.  For $g:\C^m\to\R$,
these quantities are computed by identifying $g$ with a function on
$\R^{2m}$.  

We begin with the cases of unrestricted real or complex $n \times n$
matrices.

\begin{thm}\label{T:traces-of-powers-clt-nonnormal}
  Let $X$ be a random matrix in $V = \mat{n}{\C}$ or $\mat{n}{\R}$,
  whose distribution is invariant under rotations of $V$.

  Suppose that $\E\|X\|^2=n$ and that for each $k$, there is a
  constant $\alpha_k$ depending only on $k$ such that
  \[
    t_k(X)=\left|n^{-k/2}\E\|X\|^{k}-1\right|\le\frac{\alpha_k}{n}.
  \]

  Fix $m\in\N$ with $m \ge 3$ and
\[W=(W_1,W_2,\ldots,W_m)=(\tr(X),\tr(X^2),\ldots,\tr(X^m))\in\R^m.\]
\begin{enumerate}
\item If $V=\mat{n}{\C}$ and $G$ is a standard complex Gaussian random vector in $\C^m$, then for any $f\in C^2(\C^m)$, 
\[\big|\E f(W)-\E f(\Sigma^{1/2}G)\big|\le \frac{\kappa_mM_2(f)}{n},\]
where $\Sigma$ is the diagonal matrix with $p$-$p$ entry given by
\(\sigma_{pp}=p\) and $\kappa_m$ is a positive constant depending only
on $m$ and $\alpha_1,\ldots, \alpha_m$.

\item If $V=\mat{n}{\R}$ and $G$ is a standard Gaussian random vector in $\R^m$, then for any $f\in C^2(\R^m)$, 
\[\big|\E f(W-\E W)-\E f(\Sigma^{1/2}G)\big|\le \frac{\kappa_m(M_1(f)+M_2(f))}{n},\]
where $\Sigma$ is the diagonal matrix with $p$-$p$ entry given by
\(\sigma_{pp}=p\) and $\kappa_m$ is a positive constant depending only
on $m$ and $\alpha_1,\ldots, \alpha_m$.
\end{enumerate}
\end{thm}

As in Theorem \ref{T:traces-of-powers-means}, the hypothesis on $t_k$
can be weakened somewhat, at the expense of a more complicated version
of the conclusion.

Theorem \ref{T:traces-of-powers-clt-nonnormal} immediately implies the
following.

\begin{cor}
  \label{T:polynomial-clt-nonnormal}
  For each $n$, let $X_n$ be an $n\times n$ satisfying the hypotheses
  of Theorem \ref{T:traces-of-powers-clt-nonnormal} (with constants
  $\alpha_k$ independent of $n$). Given any polynomial function,
  $g:\C \to \C$, define $X_{g,n} = \tr g(X_n)$. Then the stochastic
  process $\{X_{g,n} - n g(0) \}_g$ indexed by polynomials converges
  as $n \to \infty$, in the sense of finite dimensional distributions,
  to a centered complex-valued Gaussian process $\{Z_g\}_g$ with
  covariance given by
  \[
    \E Z_g \overline{Z_h} = \frac{1}{\pi} \int_D g'(z)
    \overline{h'(z)} \ d^2 z,
  \]
  where $D = \Set{z \in \C}{\abs{z} \le 1}$ and $d^2 z$ refers to
  integration with respect to Lebesgue measure.
\end{cor}

Results of the same form as Corollary \ref{T:polynomial-clt-nonnormal}
are proved in \cite{RiSi,CiErSc} and \cite{NoPe} for random matrices
with independent complex or real entries (satisfying some technical
conditions), respectively.  In \cite{RiSi} the test functions need not
be polynomials, but are required to be analytic on a neighborhood
of the disc of radius $4$; in \cite{CiErSc} this is weakened
substantially to $C^{2+\epsilon}$ test functions.  One might hope to
extend Corollary \ref{T:polynomial-clt-nonnormal} from polynomials to
analytic or still more general test functions by approximation (as is
done, for example, in \cite{DoSt} in the case of Haar-distributed
random unitary matrices).  However, the dependence on the constants
$\kappa_m$ in Theorem \ref{T:traces-of-powers-clt-nonnormal} on $m$
provided by our proofs is insufficient to carry out such an
approximation argument.  (Moreover, as seen in \cite{RiVi,CiErSc},
extending beyond analytic functions requires a more complicated
description of the limiting covariance structure.)

\bigskip

There are several key differences between the Hermitian case and the
case of unrestricted complex matrices, the most crucial of which is
that $W_2=\tr(X^2)=\tr(XX^*)=\|X\|^2$ when $X$ is Hermitian.  In
particular, a multivariate central limit theorem cannot hold in
general for the vector
\[W=(W_1,W_2,\ldots,W_m)\]
because the second component need not have Gaussian fluctuations.  In the case of $X$ uniformly distributed on the sphere of
radius $\sqrt{n}$ in $\symmat{n}{\C}$, $W_2$ is deterministic, and so
one could hope for a central limit theorem involving a covariance
matrix of rank $m-1$ (and indeed this is the case).

A related difference from the non-Hermitian case is that $\E W_p$ is
of order $n$ for all even $p$ in the Hermitian case; a consequence of
this fact is that it is necessary to make a stronger (though still
rather mild) concentration hypothesis for $\norm{X}$ than in Theorems
\ref{T:traces-of-powers-means} and
\ref{T:traces-of-powers-clt-nonnormal}.

\begin{thm}\label{T:traces-of-powers-clt-normal}
  Let $X$ be a random matrix in $V = \symmat{n}{\C}$ or $\symmat{n}{\R}$,
  whose distribution is invariant under rotations of $V$.

  Suppose that $\E\|X\|^2=n$ and that for each $k$, there is a
  constant $\alpha_k$ depending only on $k$ such that
  \[t_k(X)=\left|n^{-k/2}\E\|X\|^{k}-1\right|\le\frac{\alpha_k}{n^2}.\]   
  
  Fix $m\in\N$ with $m \ge 3$ and
  \[W=(W_1,W_2,\ldots,W_m)=(\tr(X),\tr(X^2),\ldots,\tr(X^m))\in\R^m.\]

  Let $Y_2=\|X\|^2-n$ and define 
  \[
    Z=(Z_1,Z_3,Z_4,\ldots,Z_m)\qquad\qquad Z_p=W_p-\E W_p-\frac{p\E
      W_p}{2n}Y_2.
  \]
  Let $\Sigma=A^{-1}B,$ where $A$ and $B$ are indexed by
  $\{1,\ldots,m\}\setminus\{2\}$ with  entries
  \[
    a_{pq}=\begin{cases}
      -2pC_{(p-2-q)/2}& \text{if $1\le q\le p-2$ and $p-q$ is even,}\\
      p& \text{if }q=p,\\
      0&\mbox{otherwise}.
    \end{cases}
  \]
  and 
  \[ 
    b_{pq}=2pq \begin{cases}
      C_{(p+q-2)/2} - C_{p/2} C_{q/2} & \text{if $p$ and $q$ are both even,} \\
      C_{(p+q-2)/2} & \text{if $p$ and $q$ are both odd,} \\
      0 & \text{if $p$ and $q$ have opposite parities}.
    \end{cases}
  \]
  and $C_r$ again denotes the $r$th Catalan number.
  Then for any $f\in C^2(\R^{m-1})$, 
  \[
    \big|\E f(Z)-\E f(\Sigma^{1/2}G)\big|\le
    \frac{\kappa_m(M_1(f)+M_2(f))}{n},
  \] 
  where $\kappa_m$ is a positive constant depending only on $m$
   and $\alpha_1,\ldots, \alpha_m$, and $G$ is a standard
  Gaussian random vector in $\R^{m-1}$.
\end{thm}

It is not obvious from the form given in the statement of Theorem
\ref{T:traces-of-powers-clt-normal} that the covariance matrix
$\Sigma$ is symmetric, let alone positive semidefinite. It will,
however, follow from the proof of Theorem
\ref{T:traces-of-powers-clt-normal} that this is indeed the case.

Theorem \ref{T:traces-of-powers-clt-normal} immediately implies a
multivariate central limit theorem for traces of \emph{odd} powers of
$X$.  It also implies a central limit theorem for traces of powers
other than $2$ if $X$ is uniformly distributed on the sphere of radius
$\sqrt{n}$ in $\symmat{n}{\C}$ or $\symmat{n}{\R}$, or more generally
if $\E \abs{Y_2} = o(1)$.  In either of those two situations, one can
deduce a result analogous to Corollary
\ref{T:polynomial-clt-nonnormal}, although with more complicated
expressions both for the means (as seen in Theorem
\ref{T:traces-of-powers-means}) and for the covariance; for brevity we
omit a precise statement here.  Analogous results for Wigner matrices,
in various levels of generality, were proved in
\cite{Johansson2,LyPa,Shcherbina,SiSo}.
 
Rotationally invariant ensembles of antihermitian matrices reduce to
the Hermitian case: if $X$ is a rotationally invariant Hermitian
random matrix, then $iX$ is a rotationally invariant antihermitian
matrix, and in particular $\tr (iX)^p = i^p \tr (X^p)$.  A version of
Theorem \ref{T:traces-of-powers-clt-normal} for antihermitian matrices
is therefore a formal consequence of Theorem
\ref{T:traces-of-powers-clt-normal} itself.  The explicit statement
will be somewhat complicated, however, since the random vector $Z$
will be distributed in a particular $(m-1)$-dimensional real subspace
of $\C^{m-1}$.

In contrast, the case of real antisymmetric matrices requires an
independent analysis.  Note in particular that if $X \in
\asymmat{n}{\R}$ then $\tr (X^p) = 0$ for every odd $p$.  We have the
following result for such random matrices.

\begin{thm}\label{T:real-asymm-clt}
  Let $X$ be a random matrix in $V = \asymmat{n}{\R}$ 
  whose distribution is invariant under rotations of $V$.

  Suppose that $\E\|X\|^2=n$ and that for each $k$, there is a
  constant $\alpha_k$ depending only on $k$ such that
  \[t_k(X)=\left|n^{-k/2}\E\|X\|^{k}-1\right|\le\frac{\alpha_k}{n^2}.\]   
  Let
  $m \ge 4$ be even and
  \[
    W=(W_4,W_6,\ldots,W_m)=(\tr(X^4),\tr(X^6),\ldots,\tr(X^m)).
  \]
  Then for each even $p$, 
  \[
    \E W_p= (-1)^{p/2} n C_{p/2}+O(1),
  \]
  where $C_{p/2}$ is the $p/2$ Catalan number.  
  
  Let $Y_2=\|X\|^2-n$ and define 
  \[
    Z=(Z_4,Z_6,\ldots,Z_m)\qquad\qquad Z_p=W_p-\E W_p-(-1)^{p/2}\frac{p\E
      W_p}{2n}Y_2.
  \]
  Let $\Sigma=A^{-1}B,$ where $A$ and $B$ have entries (for $p,q \ge 4$
  even)
  \[
    a_{pq}=
    \begin{cases}
      (-1)^{(p-q)/2}2 C_{(p-2-q)/2} & \text{if } 4\le q\le p-2,\\
      1 & \text{if } q=p, \\
      0 &\mbox{otherwise},\end{cases}
  \]
  and 
  \[
    b_{pq} = (-1)^{(p+q)/2} q (C_{(p+q-2)/2} - C_{p/2} C_{q/2}).
  \]
  
  Then for any $f\in C^2(\R^{(m-2)/2})$, 
  \[\big|\E f(Z)-\E f(\Sigma^{1/2}G)\big|\le
    \frac{\kappa_m(M_1(f) + M_2(f))}{n},\] where $\kappa_m$ is a
  positive constant depending only on $m$ and
  $\alpha_1,\ldots, \alpha_m$, and $G$ is a standard Gaussian random
  vector in $\R^{(m-2)/2}$.
\end{thm}

To our knowledge, the only previous paper whose results explicitly
include a central limit theorem for linear eigenvalue statistics of
real antisymmetric random matrices is \cite{ORoRe}, although the
methods of most previous works on Hermitian random matrices could
presumably be adapted to cover the real antisymmetric case as well.

Our results are proved using a general Gaussian approximation theorem
for exchangeable pairs \cite{EM-Stein-multi,DoSt}.  In section
\ref{S:pair} below we state the general approximation theorem, define
the exchangeable pair for an arbitrary matrix subspace $V$, and carry
out as much of the analysis as possible without specifying $V$; this
may be characterized as the essentially ``algebraic'' part of our
proofs.  The remaining sections carry out the ``asymptotic'' part of
the argument, on a case-by-case basis, for each of the subspaces $V$
considered here.  In sections \ref{S:nonsym-c} and \ref{S:nonsym-r} we
prove Theorems \ref{T:traces-of-powers-means} and
\ref{T:traces-of-powers-clt-nonnormal} for the cases of
$V = \mat{n}{\C}$ and $\mat{n}{\R}$, respectively.  In section
\ref{S:Hermitian} we prove Theorems \ref{T:traces-of-powers-means} and
\ref{T:traces-of-powers-clt-normal} for $V = \symmat{n}{\C}$.  In
section \ref{S:symmetric} we indicate how to modify the proofs of
section \ref{S:Hermitian} for $V = \symmat{n}{\R}$.  The proof of
Theorem \ref{T:real-asymm-clt} is yet another variation on the same
theme, and is omitted.

\section{Common framework: The exchangeable pair}\label{S:pair}

As discussed in the introduction, the proof of Theorem
\ref{T:traces-of-powers-means} is essentially a by-product of the
proofs of Theorems \ref{T:traces-of-powers-clt-nonnormal} and
\ref{T:traces-of-powers-clt-normal}, and so we postpone the proof of
Theorem \ref{T:traces-of-powers-means} for the moment.  The other main
theorems are proved via a version of Stein's method.  The complex form
of the multivariate infinitesimal version of Stein's method of
exchangeable pairs stated below is due to D\"obler and Stolz
\cite{DoSt}, following earlier work of E.\ Meckes
\cite{EM-Stein-multi} in the real case.

\begin{thm}\label{T:inf-abstract}
  Let $W$ be a centered random vector in $\C^m$ and, for each
  $\epsilon\in(0,1)$, suppose that $(W,W_\epsilon)$ is an exchangeable
  pair.  Let $\mathcal{G}$ be a $\sigma$-algebra with respect to which
  $W$ is measurable.  Suppose that there is an invertible matrix
  $\Lambda$, a symmetric, non-negative definite matrix $\Sigma$, a
  $\mathcal{G}$-measurable random vector $E \in \C^m$,
  $\mathcal{G}$-measurable random matrices $E', E'' \in \mat{m}{\C}$,
  and a deterministic function $s(\epsilon)$ such that
\begin{enumerate}
\item \label{inf-lincond}
$$\frac{1}{s(\epsilon)}\E\left[W_\epsilon-W\big|\mathcal{G}\right]\xrightarrow[\epsilon\to0]{L_1}-\Lambda W+E,$$
\item \label{inf-quadcond}
$$\frac{1}{s(\epsilon)}\E\left[(W_\epsilon-W)(W_\epsilon-W)^*\big|\mathcal{G}\right]\xrightarrow[\epsilon\to0]{L_1(\|\cdot\|)}2\Lambda\Sigma+E',$$
\item \label{inf-quadcond2}
$$\frac{1}{s(\epsilon)}\E\left[(W_\epsilon-W)(W_\epsilon-W)^T\big|\mathcal{G}\right]\xrightarrow[\epsilon\to0]{L_1(\|\cdot\|)}E''.$$
\item \label{inf-tricond}
For each $\rho>0$, 
$$\lim_{\epsilon\to0}\frac{1}{s(\epsilon)}\E\left[|W_\epsilon-W|^2\I(|W_\epsilon-W|^2>\rho)\right]=0.$$
\end{enumerate}

Then for $g\in C^2(\C^m)$,
\begin{equation}\begin{split}\label{inf-bd1}
    \big|\E g(W)-\E
    g(\Sigma^{1/2}Z)\big|&\le\|\Lambda^{-1}\|_{op}\left[
      M_1(g)\E|E|+\frac{\sqrt{m}}{4}M_2(g)
      \left(\E\|E'\|+\E\|E''\|\right)\right],
\end{split}\end{equation} 
where $Z$ is a standard complex Gaussian random vector in $\C^m$;
i.e., $Z_j=X_j+iY_j$, where $\{X_1,Y_1,\ldots,X_m,Y_m\}$ are i.i.d.\
$\mathcal{N}\left(0,\tfrac{1}{2}\right)$, and $\abs{E}$
denotes the Euclidean norm of the random vector $E$.
\end{thm}

\noindent \emph {Remarks:} 
\begin{enumerate} 
\item To recover the real case of Theorem \ref{T:inf-abstract}, one
  omits condition (3) and the term $\E \norm{E''}$ in
  \eqref{inf-bd1}.  The real case will be used for all the proofs
  below except for the case of $V = \mat{n}{\C}$.
  
\item In practice, we typically replace condition (4) with
  the formally stronger condition 
  \[
    \lim_{\epsilon\to0}\frac{1}{s(\epsilon)}\E|W_\epsilon-W|^3=0.
  \]
  This condition is trivially satisfied in our applications, since
  $W_\epsilon$ is constructed so that $W_\epsilon-W=\epsilon Y$ for
  some random vector $Y$ with $\E|Y|^3<\infty$.

\end{enumerate}

\bigskip

A parametrized family $(X,X_\epsilon)$ of
exchangeable pairs of random matrices can be constructed as
follows. As above, let $X = \sum_{\alpha=1}^d X_\alpha B_\alpha$,
where $\{X_\alpha\}_{\alpha = 1}^d$ is a random vector in $\R^d$ with a
rotationally invariant distribution and $\{B_\alpha\}_{\alpha = 1}^d$
is an orthonormal basis of a $d$-dimensional subspace $V$ of
$\mat{n}{\C}$. We assume that $\E \norm{X}^2 = n$ and that $\E
\norm{X}^{2m} < \infty$.

For a $d \times d$ matrix $A=\left[a_{jk}\right]_{j,k=1}^d$ in the
orthogonal group $\Orthogonal{d}$, denote by $A(X)$ the transformation
of $X$ given by
\[
A(X)=\sum_{\alpha=1}^d\left(\sum_{\beta=1}^da_{\alpha\beta}X_\beta\right)B_\alpha.
\]

Now fix $\epsilon$, and let
$$R_\epsilon=\begin{bmatrix}\sqrt{1-\epsilon^2}&\epsilon\\-\epsilon&
\sqrt{1-\epsilon^2}\end{bmatrix}\oplus I_{d-2}\in\Orthogonal{d}.$$
That is, $R_\epsilon$ represents a rotation by $\arcsin(\epsilon)$ in
the plane spanned by the first two standard basis vectors of $\R^d$.
Choose $U\in \Orthogonal{d}$ according to Haar measure, independent
of $X$, and let 
$$X_\epsilon=(UR_\epsilon U^T)(X).$$
That is, $X_\epsilon$ is a small random rotation (in matrix space) of the random matrix $X$, and so $(X,X_\epsilon)$ is exchangeable for each $\epsilon$.
For each $p\in\{1,\ldots,m\}$, define 
\[W_{\epsilon,p}:=\tr(X_\epsilon^p);\]
the $m$-dimensional random vectors $(W,W_\epsilon)$ are then exchangeable for each $\epsilon$.

To apply Theorem \ref{T:inf-abstract}, the difference $W_\epsilon-W$
must be expanded in powers of $\epsilon$.  First, 
\[UR_\epsilon U^T=U\left[I_d+\epsilon C\oplus 0_{d-2}+\left(-\frac{\epsilon^2}{2}+O(\epsilon^4)\right)I_2\oplus 0_{d-2}\right]U^T,\]
where $0_{n}$ is the $n\times n$ matrix of all zeroes, $C$ is the $2\times 2$ matrix 
\[C=\begin{bmatrix}0&1\\-1&0\end{bmatrix},\]
and the $O(\epsilon^4)$ is  the deterministic error in replacing $\sqrt{1-\epsilon^2}-1$ by $-\frac{\epsilon^2}{2}$.  Letting $K$ denote the first two columns of $U$ and $Q:=KCK^T$, we have
\[UR_\epsilon U^T=I_d+\epsilon Q+\left(-\frac{\epsilon^2}{2}+O(\epsilon^4)\right)KK^T.\]
It follows that 
\begin{equation}\begin{split}\label{E:diff-expansion}
W_{\epsilon,p}&-W_p\\&=\tr(X_\epsilon^p-X^p)\\&=\tr\left(\left[X+\epsilon Q(X)+\left(-\frac{\epsilon^2}{2}+O(\epsilon^4)\right)KK^T(X)\right]^p-X^p\right)\\&=\epsilon\sum_{j=0}^{p-1}\tr\left(X^j[Q(X)]X^{p-1-j}\right)\\&\qquad+\epsilon^2\left[\sum_{j=0}^{p-2}\sum_{k=0}^{p-2-j}\tr\left(X^j[Q(X)]X^k[Q(X)]X^{p-2-j-k}\right)-\frac{1}{2}\sum_{j=0}^{p-1}\tr\left(X^j[KK^T(X)]X^{p-1-j}\right)\right]+O(\epsilon^3)\\&=\epsilon p\tr(X^{p-1}[Q(X)])\\&\qquad+\epsilon^2\left[\sum_{\ell=0}^{p-2}(\ell+1)\tr\left(X^\ell[Q(X)]X^{p-2-\ell}[Q(X)]\right)-\frac{p}{2}\tr(X^{p-1}[KK^T(X)])\right]+O(\epsilon^3),
\end{split}\end{equation}
where the implied constant in the error $O(\epsilon^3)$ is a random
variable (with all moments finite) depending on $X$ and
$U$. (The $O(\eps^3)$
terms here and below may depend on $\E \norm{X}^{2m}$, and hence are
not necessarily uniform in either $n$ or $m$ without more assumptions
than have been made up to this point.  However, in Theorem
\ref{T:inf-abstract} the limits as $\epsilon \to 0$ are taken with $n$
and $m$ both fixed, so this poses no difficulty.)

Analyzing this expression comes down to integrals over the orthogonal
group $\Orthogonal{d}$ and over the sphere $\mathbb{S}^{d-1}$.  The
following concentration result from \cite{MeSz} plays an important
technical role. Given a
polynomial $Q(x,y)$ in two variables, we refer to a function
$P(X) = Q(X,X^*)$ on $\mat{n}{\C}$ as a $*$-polynomial.

\begin{prop}\label{T:star-poly-concentration}
Let $P$ be a $*$-polynomial of degree at most $p$, and let $X$ be a
random $n\times n$ matrix uniformly distributed in a sphere of radius
$\sqrt{n}$ in a subspace of $\mat{n}{\C}$ of dimension $d \ge c n^2$. Then
\[
\Prob[\abs{\tr P(X) - \E \tr P(X)} \ge t] \le \kappa_p \exp[-c_p \min\{t^2,
nt^{2/p}\}]
\]
and
\[
\norm{\tr P(X) - \E \tr P(X)}_q \le C_p \max\left\{\sqrt{q}, \left(\frac{q}{n}\right)^{p/2}\right\}
\]
for each $q \ge 1$.
Here $\kappa_p, c_p, C_p \ge 0$ are constants depending only on $p$
and $c$, and $\norm{Y}_q = (\E \abs{Y}^q)^{1/q}$ denotes the $L_q$
norm of a random variable.
\end{prop}

The following lemma is key  in applying Theorem \ref{T:inf-abstract}.

\begin{lemma}\label{T:limits_all_spaces} For $W$ as above and $p$
  fixed, 
\begin{enumerate}
\item \label{P:lin}\begin{equation*}\begin{split}
\lim_{\epsilon\to0}\frac{1}{\epsilon^2}&\E\left[W_{\epsilon,p}-W_p\big|X\right]\\
&=\sum_{\ell=0}^{p-2}\frac{2(\ell+1)
    \|X\|^2}{d(d-1)}\tr\left(X^\ell\sum_\alpha B_\alpha X^{p-2-\ell}B_\alpha\right)-\frac{p(p+d-2)}{d(d-1)}W_p,\end{split}\end{equation*}

\item \label{P:quad_bar}\begin{equation*}\begin{split}\lim_{\epsilon\to0}&\frac{1}{\epsilon^2}\E\left[(W_\epsilon-W)_p\overline{(W_\epsilon-W)_q}|X\right]\\&=\frac{2
        pq}{d(d-1)}\left[\|X\|^2\sum_{\alpha=1}^d
        \tr(X^{p-1}B_\alpha)\overline{\tr(X^{q-1}B_\alpha)}-\sum_{\alpha,\beta=1}^dX_\alpha
        X_\beta\tr(X^{p-1}B_\alpha)\overline{\tr(X^{q-1}B_\beta)}\right],\end{split}\end{equation*}
\item \label{P:quad_no_bar}\begin{equation*}\begin{split}\lim_{\epsilon\to0}&\frac{1}{\epsilon^2}\E\left[(W_\epsilon-W)_p(W_\epsilon-W)_q|X\right]\\&=\frac{2
        pq}{d(d-1)}\left[\|X\|^2\sum_{\alpha=1}^d
        \tr(X^{p-1}B_\alpha)\tr(X^{q-1}B_\alpha)-\sum_{\alpha,\beta=1}^dX_\alpha
        X_\beta\tr(X^{p-1}B_\alpha)\tr(X^{q-1}B_\beta)\right],\end{split}\end{equation*}
\item \label{P:cube}\[\lim_{\epsilon\to0}\frac{1}{\epsilon^2}\E|W_\epsilon-W|^3=0.\]

\end{enumerate}
In each case, the convergence is in the $L_1$ sense.
\end{lemma}

\begin{proof}
By the expansion of $W_{\epsilon,p}-W_p$ in powers of $\epsilon$ given in \eqref{E:diff-expansion}, it follows from the independence of $X$ and $U$ that
\begin{equation*}\begin{split}
\E\left[W_{\epsilon,p}-W_p\big|X\right]
&=\epsilon p\tr(X^{p-1}\E\left[Q(X)\big|X\right])\\
&\quad+\epsilon^2\left[\sum_{\ell=0}^{p-2}(\ell+1)\E\left[\left.\tr\left(X^\ell[Q(X)]X^{p-2-\ell}[Q(X)]\right)\right|X\right]\right.\\
&\qquad \qquad \left.\phantom{\sum_{\ell=0}^{p-2}}
-\frac{p}{2}\tr\left(X^{p-1}\E\left[KK^T(X)\big|X\right]\right)\right]+O(\epsilon^3),
\end{split}\end{equation*}
where here and in what follows, the implied constants in the error
term are random but bounded in $L_1$.

The entries of $KK^T$ and $Q$ are given in terms of the entries of $U=[u_{jk}]_{jk=1}^d$ by 
\[
[KK^T]_{jk}=u_{j1}u_{k1}+u_{j2}u_{k2},
\qquad\qquad[Q]_{jk}=u_{j1}u_{k2}-u_{j2}u_{k1}.
\]
From this it is easy to see that 
\[
\E[KK^T]=\frac{2}{d}I_d,\qquad\qquad\E[Q]=0,
\]
and thus
\begin{equation}\begin{split}\label{E:diff1}\E&\left[W_{\epsilon,p}-W_p\big|X\right]=\epsilon^2\left[\sum_{\ell=0}^{p-2}(\ell+1)\E\left[\left.\tr\left(X^\ell[Q(X)]X^{p-2-\ell}[Q(X)]\right)\right|X\right]-\frac{p}{d}W_p\right]+O(\epsilon^3).\end{split}\end{equation}

Now,
\[\E\left[\left.\tr\left(X^\ell[Q(X)]X^{p-2-\ell}[Q(X)]\right)\right|X\right]=\tr\left(X^\ell\E\left[\left.[Q(X)]X^{p-2-\ell}[Q(X)]\right|X\right]\right).\]
For notational convenience, write $A:=X^{p-2-\ell}$. If $q_{\alpha\beta}$ denotes the $(\alpha,\beta)$ entry of $Q$, then by expanding in the basis $\{B_j\}$,
\[[Q(X)]A[Q(X)]=\sum_{\alpha,\beta,\gamma,\delta=1}^dq_{\alpha\beta}q_{\gamma\delta}X_\beta
  X_\delta B_\alpha AB_\gamma\]
and so
\[\E\left[\left.[Q(X)]X^{p-2-\ell}[Q(X)]\right|X\right]=\sum_{\alpha,\beta,\gamma,\delta=1}^d\E\left[q_{\alpha\beta}q_{\gamma\delta}\right]X_\beta
  X_\delta B_\alpha AB_\gamma.\]
The formulae above for $q_{\alpha\beta}$ in terms of the entries of $U$ can be used to derive the following (see Lemma 9 of \cite{ChMe})
\begin{equation}\label{E:qmoments}\E\left[q_{\alpha\beta}q_{\gamma\delta}\right]=\frac{2}{d(d-1)}\left[\delta_{\alpha\gamma}\delta_{\beta\delta}-\delta_{\alpha\delta}\delta_{\beta\gamma}\right],\end{equation}
and so
\begin{equation*}\begin{split}
\E\left[\left.[Q(X)]X^{p-2-\ell}[Q(X)]\right|X\right]&=\frac{2}{d(d-1)}\left[\sum_{\alpha,\beta}X_\beta^2B_\alpha
  AB_\alpha-\sum_{\alpha,\beta}X_\alpha X_\beta B_\alpha
  AB_\beta\right]\\&=\frac{2}{d(d-1)}\left[\|X\|^2\sum_{\alpha}B_\alpha
  AB_\alpha -XAX\right].
\end{split}\end{equation*}

It thus follows from \eqref{E:diff1} that 
\begin{equation*}\begin{split}
\E&\left[W_{\epsilon,p}-W_p\big|X\right]\\&=\epsilon^2\left[\sum_{\ell=0}^{p-2}\frac{2(\ell+1)}{d(d-1)}\left(\|X\|^2\tr\left(X^\ell\sum_\alpha
      B_\alpha X^{p-2-\ell}B_\alpha\right)-W_p\right)-\frac{p}{d}W_p\right]+O(\epsilon^3)
\\&=\epsilon^2\left[\sum_{\ell=0}^{p-2}\frac{2(\ell+1)
    \|X\|^2}{d(d-1)}\tr\left(X^\ell\sum_\alpha B_\alpha X^{p-2-\ell}B_\alpha\right)-\frac{p(p+d-2)}{d(d-1)}W_p\right]+O(\epsilon^3),
\end{split}\end{equation*}
whence the statement of part \ref{P:lin} of the lemma.

For part \ref{P:quad_bar}, again using the expansion of $W_\epsilon-W$ in
\eqref{E:diff-expansion} yields
\begin{equation*}\begin{split}\E&\left[(W_\epsilon-W)_p\overline{(W_\epsilon-W)_q}|X\right]\\&\qquad=\epsilon^2pq\E\left[\left.\tr(X^{p-1}[Q(X)])\overline{\tr(X^{q-1}[Q(X)])}\right|X\right]+O(\epsilon^3)\\&\qquad=\epsilon^2 pq\E\left[\left.\tr\left(\sum_{\alpha,\beta=1}^dq_{\alpha\beta}X_\beta X^{p-1}B_\alpha\right) \overline{\tr\left(\sum_{\gamma,\delta=1}^dq_{\gamma\delta}X_\delta X^{q-1} B_\gamma\right)}\right|X\right]+O(\epsilon^3)\\&\qquad=\epsilon^2 pq\sum_{\alpha,\beta,\gamma,\delta=1}^d\E[q_{\alpha\beta}q_{\gamma\delta}]X_\beta X_\delta\tr(X^{p-1}B_\alpha)\overline{\tr(X^{q-1}B_\gamma)} +O(\epsilon^3).\end{split}\end{equation*}
Making use of the moment formula for $Q$ given in \eqref{E:qmoments}
then gives that
\begin{equation*}\begin{split}\E&\left[(W_\epsilon-W)_p\overline{(W_\epsilon-W)_q}|X\right]\\&=\frac{2\epsilon^2
      pq}{d(d-1)}\left[\sum_{\alpha,\beta=1}^dX_\beta^2
      \tr(X^{p-1}B_\alpha)\overline{\tr(X^{q-1}B_\alpha)}-\sum_{\alpha,\beta=1}^dX_\alpha
      X_\beta\tr(X^{p-1}B_\alpha)\overline{\tr(X^{q-1}B_\beta)}\right]+O(\epsilon^3).\end{split}\end{equation*}

Exactly the same argument for part \ref{P:quad_no_bar} gives that
\begin{equation*}\begin{split}\E&\left[(W_\epsilon-W)_p(W_\epsilon-W)_q|X\right]\\&=\frac{2\epsilon^2
      pq}{d(d-1)}\left[\sum_{\alpha,\beta=1}^dX_\beta^2
      \tr(X^{p-1}B_\alpha)\tr(X^{q-1}B_\alpha)-\sum_{\alpha,\beta=1}^dX_\alpha
      X_\beta\tr(X^{p-1}B_\alpha)\tr(X^{q-1}B_\beta)\right]+O(\epsilon^3).\end{split}\end{equation*} 

Finally, it is clear from the expansion in $\epsilon$ that 
\[\E|W_\epsilon-W|^3=O(\epsilon^3),\]
which completes the proof.
\end{proof}

At this point in the analysis, it is necessary to consider the various
subspaces separately; this is carried out in the following sections.

\section{Rotationally invariant ensembles in $\mat{n}{\C}$}\label{S:nonsym-c}

We begin with the following technical lemma. 

\begin{lemma}\label{T:prelim-nosymm-c}
  Let $X$ be a random matrix in $\mat{n}{\C}$ whose distribution is
  invariant under rotations within $\mat{n}{\C}.$ Suppose that
  $\E\|X\|^2=n$ and that for each $k$, there is a constant $\alpha_k$
  depending only on $k$ such that
  \[
    t_k(X)=\left|n^{-k/2}\E\|X\|^{k}-1\right|\le\frac{\alpha_k}{n}.
  \] 
  Then for $p,q\in\N$,
  \[
    \E\left[\|X\|^2\tr(X^p(X^*)^{q})\right]=\begin{cases}n^2+O(n),&p=q;\\0,&\mbox{otherwise}.\end{cases}
  \]
Here, the implied constant in the $O(n)$ may depend on $p,q,$ and the
constants $\alpha_k$.
\end{lemma}

\begin{proof}
  For $p\neq q$, $\E\left[\|X\|^2\tr(X^{p}(X^{q})^*)\right]=0$ by
  symmetry.  We suppose from now on that $p=q$.

  By the rotational invariance of $X$, we can write
  $X = \frac{\norm{X}}{\sqrt{n}}\widetilde{X}$, where $\widetilde{X}$ is uniformly distributed
  on the sphere of radius $\sqrt{n}$ in $\mat{n}{\C}$ and $\widetilde{X}$ is
  independent from $\norm{X}$.  We then have
  \begin{align*}
    \E\left[\|X\|^2\tr(X^{p}(X^*)^{p})\right]
    & = \left( \frac{\E \norm{X}^{2p+2}}{n^{p+1}}\right)
      n \E \tr(\widetilde{X}^{p}(\widetilde{X}^*)^{p}) ,
  \end{align*}
  and thus
  \begin{equation} \label{Eq:uniform-to-general}
    \abs{\E\left[\|X\|^2\tr(X^{p}(X^*)^{p})\right] 
    - n \E \tr(\widetilde{X}^{p}(\widetilde{X}^*)^{p}) }
     \le n t_{2p+2}(X) \E \tr(\widetilde{X}^{p}(\widetilde{X}^*)^{p}).
  \end{equation}
  It therefore suffices to prove the lemma under the assumption that
  $X$ is uniformly distributed on the sphere of radius $\sqrt{n}$ in
  $\mat{n}{\C}$; the general case follows from
  \eqref{Eq:uniform-to-general} and the assumption on $t_k(X)$. 
  
  Making this assumption, we now consider the expansion
  \begin{equation}\label{E:Eentry-explicit}
    \E \tr(X^{p}(X^*)^{p})
    = \sum_{i_1,\ldots,i_{2p}}\E\left[x_{i_1i_2}x_{i_2i_3}\cdots
      x_{i_{p}i_{p+1}}\overline{x_{i_{p+2}i_{p+1}}}\cdots
      \overline{x_{i_1i_{2p}}}\right].
  \end{equation}
  By rotational symmetry, a term on the right side of
  \eqref{E:Eentry-explicit} is non-zero only if each $x_{ij}$ appears
  the same number of times as $\overline{x_{ij}}$.  Consider the
  contribution to the sum such that $i_1,\ldots,i_{p+1}$ are distinct,
  and $i_2=i_{2p}$, $i_3=i_{2p-1}$, and so on.  The contribution of
  such terms is 
  \begin{align*}
    n(n-1)\cdots(n-p)
    &\E\left[|x_{11}|^2|x_{12}|^2\cdots|x_{1p}|^2\right]
      = n^p \frac{n (n-1) \cdots (n-p)}{(n^2+p-1) \cdots n^2}
    = n + O(1)
  \end{align*}
  making use of the standard formula for integrating polynomials over
  the sphere (see, e.g., Lemma 14 of \cite{MM-quantum}).
  
  The sum of remaining terms of \eqref{E:Eentry-explicit} is $O(1)$,
  since they necessarily involve the choice of fewer indices from
  $\{1,\ldots,n\}$, while the expectations on the right hand side
  which appear all have the same order in $n$ (this is immediate from
  the formula in \cite{MM-quantum}).  By \eqref{Eq:uniform-to-general}
  this completes the proof of the lemma.
\end{proof}

\begin{proof}[Proof of Theorems \ref{T:traces-of-powers-means} and \ref{T:traces-of-powers-clt-nonnormal} for $V=\mat{n}{\C}$]

Recall that in this context, the orthonormal basis
$\{B_\alpha\}_{\alpha=1}^d$ is $\{E_{jk}\}_{1\le j,k\le n}\cup
\{iE_{jk}\}_{1\le j,k\le n}$.  It follows that for $A\in\mat{n}{\C}$, 
\[\sum_{\alpha=1}^dB_\alpha AB_\alpha=\sum_{j,k=1}^nE_{jk}AE_{jk}+\sum_{j,k=1}^n(iE_{jk})A(iE_{jk})=0.\]

Part \ref{P:lin} of Lemma \ref{T:limits_all_spaces} then implies that
\begin{equation}\label{E:cmat-linearity}
\lim_{\epsilon \to 0}
\frac{1}{\epsilon^2}\E\left[W_{\epsilon,p}-W_p\big|X\right]=-\frac{p(p+d-2)}{d(d-1)}W_p.
\end{equation}

Note that by taking expectations of both sides of Equation
\ref{E:cmat-linearity}, the exchangeability of $(W_p,W_{\epsilon,p})$  implies that $\E W_p=0$ for all $p$; this is
also apparent from symmetry considerations (and hence the
$\mat{n}{\C}$ case of Theorem \ref{T:traces-of-powers-means}).

Equation \ref{E:cmat-linearity}  shows that the matrix $\Lambda$ in the statement of Theorem
\ref{T:inf-abstract} may be taken to be diagonal, with $(p,p)$ entry
given by $\frac{p(p+d-2)}{d(d-1)}$, and that the random vector $E=0$.  In particular, 
\[\norm{\Lambda^{-1}}_{op}=d.\] 

Next, consider part \ref{P:quad_bar} of Lemma \ref{T:limits_all_spaces}.
Let $\inprod{A}{B}_{HS}$ denote the complex Hilbert--Schmidt inner
product $\inprod{A}{B}_{HS}=\tr(AB^*)$. Since
$\{B_\alpha\}_{\alpha=1}^d=\{E_{jk}\}_{j,k=1}^n\cup\{iE_{jk}\}_{j,k=1}^n$, 
\begin{align*}
\sum_\alpha\tr(X^{p-1}B_\alpha)\overline{\tr(X^{q-1}B_\alpha)}
&=2\sum_{j,k=1}^n \tr(X^{p-1}E_{jk})\overline{\tr(X^{q-1}E_{jk})}\\
&=2\sum_{j,k=1}^n \inprod{X^{p-1}}{E_{kj}}_{HS} \overline{\inprod{X^{q-1}}{E_{kj}}_{HS}}\\
&=2\inprod{X^{p-1}}{X^{q-1}}_{HS}\\
&=2\tr(X^{p-1} (X^{q-1})^*),
\end{align*}
where the third equality follows from the fact that $\{E_{kj}\}_{j,k=1}^n$ is an orthonormal basis for the complex inner product $\inprod{\cdot}{\cdot}_{HS}$.
Similarly,
\begin{align*}
\sum_{\alpha=1}^dX_\alpha\tr(X^{p-1}B_\alpha)
&=\sum_{j,k=1}^n\inprod{X}{E_{jk}}\tr(X^{p-1}E_{jk})+\sum_{j,k=1}^n\inprod{X}{iE_{jk}}\tr(X^{p-1}iE_{jk})\\
&=\sum_{j,k=1}^n\left[\inprod{X}{E_{jk}}+i\inprod{X}{i E_{jk}}\right]\tr(X^{p-1}E_{jk}) \\
&=\sum_{j,k=1}^n \inprod{X}{E_{jk}}_{HS} \overline{\inprod{(X^{p-1})^*}{E_{jk}}_{HS}}\\
&= \inprod{X}{(X^{p-1})^*}_{HS} = \tr(X^p).
\end{align*}

It therefore follows from Lemma \ref{T:limits_all_spaces} that 
\begin{equation}\label{E:cnosymm-quadcond1}
\lim_{\epsilon\to 0} \frac{1}{\epsilon^2}
\E[(W_\epsilon-W)_p\overline{(W_\epsilon-W)_q}|X]
=\frac{2pq}{d(d-1)}\left(2\norm{X}^2\tr(X^{p-1}(X^{q-1})^*)-W_p\overline{W_q}\right).
\end{equation}

Note that if $p\neq q$, the expectation of both terms on the right is zero by symmetry.  If $p=q$, then taking expectations of both sides of
\eqref{E:cnosymm-quadcond1} gives that
\begin{align*}
2\E[\|X\|^2\tr(X^{p-1}(X^*)^{p-1})]-\E|W_p|^2&=\lim_{\epsilon\to0}\frac{d(d-1)}{2p^2\epsilon^2}\E[|(W_\epsilon-W)_p|^2]\\&=\lim_{\epsilon\to0}\frac{-d(d-1)}{p^2\epsilon^2}\E[(W_\epsilon-W)_p\overline{W_p}]\\&=\lim_{\epsilon\to0}\frac{-d(d-1)}{p^2\epsilon^2}\E\Big[\E\big[(W_\epsilon-W)_p\big|W\big]\overline{W_p}\Big]\\&=\frac{p+d-2}{p}\E|W_p|^2,
\end{align*}
where the second line follows by exchangeability and the last line
follows from  formula \eqref{E:cmat-linearity} for
$\E[(W_\epsilon-W)_p|W]$.
Since $d=2n^2$, combining this computation with Lemma \ref{T:prelim-nosymm-c} means that 
\begin{equation}
  \label{E:var-bound}
\E|W_p|^2=\frac{2p}{2p+2n^2-2}\E\left[\|X\|^2\tr(X^{p-1}(X^*)^{p-1})\right]=p+O\left(\frac{1}{n}
  \right),
\end{equation}
and then by Equation \eqref{E:cnosymm-quadcond1}, 
\[
\lim_{\epsilon\to 0} \frac{1}{\epsilon^2}
\E[(W_\epsilon-W)_p\overline{(W_\epsilon-W)_q}]
  = \left(\frac{2p^2(p+d-2)}{d(d-1)}+ O\left(\frac{1}{n^3}\right)
  \right)\delta_{pq}.
\]

We define $\Sigma$ to be the diagonal matrix with
$\sigma_{pp}=p$. Taking  $\mathcal{G} = \sigma(X)$ in Theorem \ref{T:inf-abstract}, the random matrix $E'$ then has $(p,q)$ entry
\begin{align*}
[E']_{pq}  & = \frac{2pq}{d(d-1)} \left[ 2 \norm{X}^2 \tr
             \bigl(X^{p-1} (X^{q-1})^*\bigr) - W_p
             \overline{W_q}\right] - \frac{2p^2 (p+d-2)}{d(d-1)}
             \delta_{pq} \\
& =\frac{2pq}{d(d-1)}\Big[2\|X\|^2 \tr(X^{p-1}(X^{q-1})^*)-\E\left[2\|X\|^2\tr(X^{p-1}(X^{q-1})^*)\right]\\
&\qquad \qquad \qquad -W_p\overline{W_q}+\E \left(W_p\overline{W_q}\right)\Bigl]+O\left(\frac{1}{n^3}\right)\delta_{pq}.
\end{align*}
We will estimate the expected Hilbert--Schmidt norm by
\[
  \E \norm{E'} \le \E \sum_{p,q=1}^m \abs{[E']_{pq}}.
\]

We first have
\[
\E\abs{W_p\overline{W_q}-\E \left(W_p\overline{W_q}\right)}\le 2 \E \abs{W_p\overline{W_q}} \le 2 \sqrt{\E
  \abs{W_p}^2} \sqrt{\E \abs{W_q}^2} = pq + O\left(\frac{1}{n}\right)
\]
by the Cauchy--Schwarz inequality and \eqref{E:var-bound}.

As in the proof of Lemma \ref{T:prelim-nosymm-c}, we write
$X = \frac{\norm{X}}{\sqrt{n}}\widetilde{X}$, where $\widetilde{X}$ is uniformly distributed
on the sphere of radius $\sqrt{n}$ in $\mat{n}{\C}$ and $\widetilde{X}$ is
independent from $\norm{X}$. We then have
\begin{align*}
  \E\Big|\|X\|^2
  &\tr(X^{p-1}(X^{p-1})^*)-\E\left[\|X\|^2
    \tr(X^{p-1}(X^{p-1})^*)\right]\Big|\\
  &=n\E\left|\frac{\|X\|^{2p}}{n^{p}}\tr(\widetilde{X}^{p-1}(\widetilde{X}^{p-1})^*)-\left(\E\frac{\|X\|^{2p}}{n^{p}}\right)\E
    \tr(\widetilde{X}^{p-1}(\widetilde{X}^{p-1})^*)\right|\\
  &\le n\E \tr(\widetilde{X}^{p-1}(\widetilde{X}^{p-1})^*) \E\left|\frac{\|X\|^{2p}}{n^{p}}-\E
    \frac{\|X\|^{2p}}{n^{p}}\right| \\
  & \qquad +n\left(\E
    \frac{\|X\|^{2p}}{n^{p}}\right) \E \abs{\tr(\widetilde{X}^{p-1}(\widetilde{X}^{p-1})^*) - \E \tr(\widetilde{X}^{p-1}(\widetilde{X}^{p-1})^*)}\\
  &\le  n \E \tr(\widetilde{X}^{p-1}(\widetilde{X}^{p-1})^*)\left(\sqrt{t_{4p}(X)-2t_{2p}(X)}+ t_{2p} (X) \right) \\
  & \qquad + n (1 + t_{2p}(X)) \E \abs{\tr(\widetilde{X}^{p-1}(\widetilde{X}^{p-1})^*) - \E \tr(\widetilde{X}^{p-1}(\widetilde{X}^{p-1})^*)}. 
\end{align*}
Lemma \ref{T:prelim-nosymm-c} implies that
\[
\E \tr(\widetilde{X}^{p-1}(\widetilde{X}^{p-1})^*)=n+O(1),
\]
and Proposition \ref{T:star-poly-concentration} implies that
\[
\E \abs{\tr(\widetilde{X}^{p-1}(\widetilde{X}^{p-1})^*) - \E \tr(\widetilde{X}^{p-1}(\widetilde{X}^{p-1})^*)} \le \kappa_p.
\]
We therefore have
\begin{align*}\E\Big|\|X\|^2&\tr(X^{p-1}(X^{p-1})^*)-\E\left[\|X\|^2
                              \tr(X^{p-1}(X^{p-1})^*)\right]\Big|\\&\le \kappa
                              n^2 \left(\sqrt{t_{4p}(X)-2t_{2p}(X)}+ t_{2p} (X) \right) +\kappa_pn(1+t_{2p}(X))=O(n).
\end{align*}

Similarly, recalling that when $p\neq q$ the means are $0$,
\begin{align*}
  \E \abs{\norm{X}^2 \tr \bigl(X^{p-1}(X^{q-1})^*\bigr)}
  & = n \E \abs{\frac{\norm{X}^{p+q}}{n^{(p+q)/2}}} \E \abs{\tr
    \bigl(\widetilde{X}^{p-1} (\widetilde{X}^{q-1})^*\bigr)}
    \le \kappa n (1 + t_{p+q}(X)) = O(n).
\end{align*}

Making
use of the fact that $\|\Lambda^{-1}\|_{op}=d=2n^2$, it now follows
that
\[
\|\Lambda^{-1}\|_{op}\E\|E'\|\le \frac{\kappa_m}{n}
\]
for some constant $\kappa_m$ depending only on $m$.

Finally, consider part \ref{P:quad_no_bar} of Lemma \ref{T:limits_all_spaces}.
Observe that
\begin{align*}
\sum_\alpha&\tr(X^{p-1}B_\alpha)\tr(X^{q-1}B_\alpha)\\&=\sum_{j,k=1}^n\tr(X^{p-1}E_{jk})\tr(X^{q-1}E_{jk})-\sum_{j,k=1}^n\tr(X^{p-1}E_{jk})\tr(X^{q-1}E_{jk})=0,
\end{align*}
and from above,
\begin{align*}
\sum_{\alpha=1}^dX_\alpha\tr(X^{p-1}B_\alpha)=\tr(X^p).
\end{align*}
It thus follows from Lemma \ref{T:limits_all_spaces} that 
\[
[E'']_{p,q}=\lim_{\epsilon\to0}\frac{1}{\epsilon^2}\E[(W_\epsilon-W)_p(W_\epsilon-W)q|X]=\frac{2pq}{d(d-1)}W_pW_q.
\]
By the Cauchy--Schwarz inequality and \eqref{E:var-bound},
$\E |W_pW_q|$ is bounded independent of $n$, and so
\[\E\|E''\|\le\frac{\kappa_m}{d(d-1)},\]
where $\kappa_m$ is a constant depending only on $m$. This completes the
proof of Theorem \ref{T:traces-of-powers-clt-nonnormal} in the case of $\mat{n}{\C}$.
\end{proof}

\section{Rotationally invariant ensembles in $\mat{n}{\R}$}\label{S:nonsym-r}

As in the previous section, we begin with a technical lemma.

\begin{lemma}\label{T:traces-of-products}
  Let $X$ be a random matrix in $\mat{n}{\R}$ whose distribution is
  invariant under rotations in $\mat{n}{\R}$.  Suppose that
  $\E\|X\|^2=n$ and that for each $k$, there is a constant $\alpha_k$
  such that
  \[
    t_k(X)=\E\Big|n^{-k/2}\|X\|^k-1\Big|\le\frac{\alpha_k}{n}.
  \] 
  Then for all $p,q\in \N$,
  \[
    \E\left[\|X\|^2\tr(X^p(X^T)^q)\right]=\begin{cases}n^2+O(n)&
      \text{if }p=q,
      \\O(n)& \text{if $p\neq q$ and $p-q$ is even},\\0&
      \text{if $p-q$ is odd};\end{cases}
  \] 
  where the implied constants depend on $p$, $q$, and the $\alpha_k$.
\end{lemma}

\begin{proof}
  First note that if $p - q$ is odd, then
  $\E\left[\|X\|^2\tr(X^{p}(X^T)^{q})\right]=0$ by symmetry.

  If $p-q$ is even, then as in the proof of Lemma
  \ref{T:prelim-nosymm-c} we may first assume that $X$ is uniformly
  distributed on the sphere of radius $\sqrt{n}$ in $\mat{n}{\R}$.
  We have the expansion
  \begin{equation}\label{Eq:real-expansion}
    \E\left[\tr(X^{p}(X^T)^{q})\right]= 
    \sum_{\substack{i,j\\i_1,\ldots,i_{p-1}\\j_1,\ldots,j_{q-1}}}
    \E\left[x_{ii_1}x_{i_1i_2}\cdots
      x_{i_{p-1}j}x_{ij_1}x_{j_1j_2}\cdots x_{j_{q-1}j}\right]
  \end{equation}
  Consider first the case that $p=q$.  A term on the right side of
  \eqref{Eq:real-expansion} is nonzero only if each matrix entry
  $x_{ij}$ appears an even number of times. The total contribution
  from terms in which the indices (including $i$ and $j$) are chosen
  such that $i_1=j_1,\ldots,i_{p-1}=j_{p-1}$, but are otherwise
  distinct, is
  \begin{align*}
    n(n-1)\cdots\left(n-p\right) \E\left[x_{11}^2\cdots
    x_{1p}^2\right] = 
    n^p \frac{n(n-1)\cdots\left(n-p\right)}{(n^2+2p-2)(n^2+2p-4)\cdots
    n^2}
    & =n+O(1).
  \end{align*}
  Each of the non-zero expectations has the same order in $n$, and so
  this is the main contribution to the sum, since it involves the
  maximum number of distinct indices.
  
  Now suppose that $p\neq q$ and $p-q$ is even; assume without loss of
  generality that $p< q$ and write $q = p + 2k$.  Consider the
  contribution of those terms on the right side of
  \eqref{Eq:real-expansion} in which $i_\ell=j_\ell$ for
  $1\le \ell\le p-1$, $j_\ell=j_{k+\ell}$ for $p\le\ell\le p+k-1$, and
  $j = j_p = j_{q-1}$, and the indices are distinct except for these
  restrictions. The contribution of these terms is 
  \begin{align*}
    n(n-1)\cdots\left(n-p-k+1\right)
    \E\left[x_{11}^2\cdots x_{1,p+k}^2\right]
    &
      =n^{(p+q)/2}\frac{n(n-1)\cdots\left(n-\frac{p+q}{2}+1\right)}{(n^2+p+q-2)(n^2+p+q-4)\cdots n^2} \\
    & =1+O\left(n^{-1}\right).
  \end{align*}

  The leading contribution to \eqref{Eq:real-expansion} is made by
  these terms, and others obtained by permuting the equality structure
  among the indices $j, j_p, \dots, j_{q-1}$; these equality
  structures maximize the number of indices which can be chosen to be
  distinct in this group, and all nonzero terms are of the same order
  in $n$.  Since we are not interested in the leading coefficient in
  \eqref{Eq:real-expansion} in this case, it suffices for our purposes
  to note that the number of permutations is bounded in terms of $p$
  and $q$.
\end{proof}

We now proceed with the proofs of the $\mat{n}{\R}$ cases of Theorems
\ref{T:traces-of-powers-means} and \ref{T:traces-of-powers-clt-nonnormal}.

\begin{proof}[Proof of Theorem \ref{T:traces-of-powers-means} for $V=\mat{n}{\R}$]
 Trivially, if $p$ is odd then $\E W_p=0$. 

To treat the case that $p$ is even, we make use of Lemma
\ref{T:limits_all_spaces}.  Recall that in $\mat{n}{\R}$, the orthonormal basis
 $\{B_\alpha\}_{\alpha=1}^d=\{E_{jk}\}_{j,k=1}^n$, so that given $A\in\mat{n}{\R}$, 
\[
\sum_{\alpha=1}^dB_\alpha AB_\alpha =\sum_{j,k=1}^nE_{jk}AE_{jk}=A^T.
\]
It follows from this computation and part \ref{P:lin} of Lemma
\ref{T:limits_all_spaces} that
\begin{equation}\label{E:lin-diff-real-nosym}\begin{split}
  \lim_{\epsilon\to0}\frac{1}{\epsilon^2}\E\left[W_{\epsilon,p}-W_p\big|X\right]
  &=\sum_{\ell=0}^{p-2}\frac{2\|X\|^2(\ell+1)
   }{d(d-1)}\tr\left(X^\ell(X^T)^{p-2-\ell}\right)-\frac{p(p+d-2)}{d(d-1)}W_p \\
  &=\sum_{\ell=0}^{p-2}\frac{\|X\|^2 p}{d(d-1)}\tr\left(X^\ell(X^T)^{p-2-\ell}\right)-\frac{p(p+d-2)}{d(d-1)}W_p,
 \end{split}
\end{equation}
where the second line follows by replacing $\ell$ with $p-2-\ell$, and
averaging the resulting expression with the first.

Since the expectation of the left-hand side of \eqref{E:lin-diff-real-nosym} is zero by
exchangeability, the expectation of the right-hand side is zero as well, and so taking the expectation of both sides of the formula above gives that
\[
\E
W_p=\sum_{\ell=0}^{p-2}\frac{1}{(p+d-2)}\E\left[\|X\|^2\tr\left(X^\ell(X^T)^{p-2-\ell}\right)\right].
\]
By Lemma \ref{T:traces-of-products},
\[
\E\left[\|X\|^2\tr\left(X^{\frac{p}{2}-1}(X^T)^{\frac{p}{2}-1}\right)\right]=n^2+O(n),
\] 
and all the other terms in the above sum are $O(n)$, and thus
$\E W_p=1+O(n^{-1})$.
\end{proof}

\begin{proof}[Proof of Theorem \ref{T:traces-of-powers-clt-nonnormal} for
  $V=\mat{n}{\R}$]
We begin with condition \emph{(\ref{inf-lincond})} of Theorem
\ref{T:inf-abstract}.  Starting from Equation
\eqref{E:lin-diff-real-nosym}  above, since both sides of the equation
have mean zero, it follows that if $Y_p:=W_p-\E W_p$ and
$Y_{\epsilon,p}:=W_{\epsilon,p}-\E W_{\epsilon,p}$, then 
\begin{equation*}\begin{split}
\lim_{\epsilon\to0}\frac{1}{\epsilon^2}&\E\left[Y_{\epsilon,p}-Y_p\big|X\right]\\
&=\frac{p}{d(d-1)}\sum_{\ell=0}^{p-2}\left[\|X\|^2\tr\left(X^\ell(X^T)^{p-2-\ell}\right)
  -\E\left(\|X\|^2
    \tr\left(X^\ell(X^T)^{p-2-\ell}\right)\right)\right]\\
&\qquad\qquad-\frac{p(p+d-2)}{d(d-1)}Y_p.
\end{split}\end{equation*}
It follows essentially as in the previous section that for any $\ell$,
\begin{align*}\E\Big|\|X\|^2&\tr(X^{\ell}(X^{T})^{p-2-\ell})-\E\left[\|X\|^2 \tr(X^{\ell}(X^{T})^{p-2-\ell})\right]\Big|=O\big(n\big),
\end{align*}
and we therefore choose the matrix $\Lambda$ in the statement of Theorem \ref{T:inf-abstract} to be diagonal with $p$th entry given by $\frac{p(p+d-2)}{d(d-1)}$, the function $s(\epsilon)=\epsilon^2$, and the error $E$ to have $p$th entry
\[E_p=\frac{p}{d(d-1)}\left[\|X\|^2\tr\left(\sum_{\ell=0}^{p-2}X^\ell(X^T)^{p-2-\ell}\right)-\E\left[\|X\|^2 \tr\left(\sum_{\ell=0}^{p-2}X^\ell(X^T)^{p-2-\ell}\right)\right]\right],\]
so that 
\[\|\Lambda^{-1}\|_{op}\E|E|\le \frac{\kappa_m}{n},\]
with the constant $\kappa_m$ depending only on $m$.

Moving on to condition \emph{(\ref{inf-quadcond})} of Theorem \ref{T:inf-abstract}, it follows from Lemma \ref{T:limits_all_spaces} exactly as in the previous case that
\[\lim_{\epsilon\to0}\frac{1}{\epsilon^2}\E\left[(W_\epsilon-W)_p(W_\epsilon-W)_q|X\right]=\frac{2
        pq}{n^2(n^2-1)}\left[\|X\|^2\tr((X^T)^{p-1}X^{q-1})-W_pW_q\right].\]

It follows from Proposition \ref{T:star-poly-concentration} and the fact that $|\E W_p|=O(1)$ that $\E|W_pW_q|\le \kappa_{p,q}$ for some constant depending only on $p$ and $q$, and so if we choose $\Sigma$ to be diagonal with $\sigma_{pp}=p$,  
the random matrix $E'$ in the statement of Theorem
\ref{T:inf-abstract} has $p$-$q$ entry
\begin{align*}[E']_{pq}=\frac{2pq}{n^2(n^2-1)}&\left[2\|X\|^2\tr((X^T)^{p-1}X^{q-1})-\E \left[2\|X\|^2\tr((X^T)^{p-1}X^{q-1})\right]\right.\\&\qquad\qquad\left.-W_pW_q+\E W_pW_q+O(n)\right].\end{align*}
By Proposition \ref{T:star-poly-concentration}, 
$\E[|W_p|^2-\E|W_p|^2]^2$ is  bounded independently of $n$, and we have observed already that
\begin{align*}\E\Big|\|X\|^2&\tr((X^T)^{p-1}X^{q-1})-\E\left[\|X\|^2\tr((X^T)^{p-1}X^{q-1})\right]\Big|=O(n)\end{align*} and so
(making use of the fact that $\|\Lambda^{-1}\|_{op}=d=n^2$),
\[\|\Lambda^{-1}\|_{op}\E\|E'\|\le \frac{\kappa'_m}{n}\]
for some constant $\kappa_m'$ depending only on $m$.  
\end{proof}

\section{Rotationally invariant ensembles in $\symmat{n}{\C}$}\label{S:Hermitian}

We initially proceed via Lemma \ref{T:limits_all_spaces} as above.
Since in $\symmat{n}{\C}$, the orthonormal basis is
\[\{B_\alpha\}_{\alpha=1}^d=\{E_{jj}\}_{j=1}^n\cup\{F_{jk}\}_{1\le
  j<k\le n}\cup\{iG_{jk}\}_{1\le j<k\le n},\]
for a given $A\in\symmat{n}{\C}$, 
\begin{equation*}\begin{split}\sum_{\alpha=1}^dB_\alpha AB_\alpha&=\sum_{j=1}^nE_{jj}AE_{jj}+\frac{1}{2}\sum_{1\le
    j<k\le
    n}(E_{jk}+E_{kj})A(E_{jk}+E_{kj}) \\&\qquad\qquad-\frac{1}{2}\sum_{1\le
    j<k\le
    n}(E_{jk}-E_{kj})A(E_{jk}-E_{kj})\\&=\sum_{j,k=1}^nE_{jk}AE_{kj}\\&=\tr(A)I.\end{split}\end{equation*}
It thus follows from Lemma \ref{T:limits_all_spaces} that
\begin{equation}\begin{split}\label{E:csym-linear1}
    \lim_{\epsilon\to0}\frac{1}{\epsilon^2}
    \E\left[W_{\epsilon,p}-W_p\big|X\right] 
    & =\sum_{\ell=0}^{p-2}\frac{2(\ell+1)}{d(d-1)} W_2W_\ell
    W_{p-2-\ell}-\frac{p(p+d-2)}{d(d-1)}W_p \\
    & = \frac{p}{d(d-1)} W_2 \sum_{\ell=0}^{p-2}W_\ell
    W_{p-2-\ell}-\frac{p(p+d-2)}{d(d-1)}W_p,
\end{split}\end{equation}
where the second line follows by replacing $\ell$ with $p-2-\ell$, and
averaging the resulting expression with the first.

We first use this expression to prove Theorem \ref{T:traces-of-powers-means}.

\begin{proof}[Proof of Theorem \ref{T:traces-of-powers-means} for $V=\symmat{n}{\C}$.]
If $p$ is odd, then $\E W_p=\E[\tr(X^p)]=0$ by symmetry.

Suppose now that $p$ is even.
As in the proofs of Lemmas \ref{T:prelim-nosymm-c} and
\ref{T:traces-of-products}, we may assume that $X$ is uniformly
distributed in the sphere of radius $\sqrt{n}$ in $\symmat{n}{\C}$, so
that $W_2 = n$ is constant. 
Equation \eqref{E:csym-linear1} and the fact that
$(W_{\epsilon, p}, W_p)$ is exchangeable imply that
\begin{equation}
  \label{Eq:EW_p-symm-c}
\E W_p
=\frac{n}{p+d-2}\sum_{\ell=0}^{p-2}\E[W_{\ell}W_{p-\ell-2}].
\end{equation}

Proposition
\ref{T:star-poly-concentration} implies that
\begin{equation}
  \label{Eq:W-covariance-bound-hermitian}
\E [W_{\ell}W_{p-\ell-2}] - (\E W_{\ell})(\E W_{p-\ell-2}) =
\cov(W_{\ell},W_{p-2-\ell}) 
\le \sqrt{\var(W_{\ell}) \var(W_{p-2-\ell})} = O(1),
\end{equation}
and so by \eqref{Eq:EW_p-symm-c},
\begin{align*}
\frac{\E W_p}{n} & = \frac{n^2}{p+d-2}\left(\sum_{\ell=0}^{p-2}\frac{\E
  W_{\ell}}{n}\frac{\E W_{p-\ell-2}}{n} +  O(n^{-2})\right).
\end{align*}

Writing $p = 2r$ and $\beta_r = \frac{\E W_{2r}}{n}$, we therefore
have that $\beta_0 = \beta_1 = 1$ and
\[
  \beta_r = \left(\sum_{k=0}^{r-1} \beta_k \beta_{r-k-1} +  O(n^{-2})\right)(1 + O(n^{-2}))
\]
for $r \ge 2$.  Recalling that the Catalan numbers
$C_r = \frac{1}{r+1}\binom{2r}{r}$ satisfy the recurrence $C_0 = 1$
and $C_r = \sum_{k=0}^{r-1} C_k C_{r-k-1}$, it now follows by induction on
$r$ that $\beta_r = C_r + O(n^{-2})$, where the $O$ term may also
depend on $r$.
\end{proof}

Note that if $X$ is uniformly distributed in the sphere of
$\symmat{n}{\C}$, then $iX$ is uniformly distributed in the sphere of
$\asymmat{n}{\C}$.  The anti-Hermitian case of Theorem
\ref{T:traces-of-powers-means} thus follows immediately from the
Hermitian case.

Recall that if $X$ is uniformly distributed on the sphere of
radius $\sqrt{n}$ in $\symmat{n}{\C}$, it follows from Proposition
\ref{T:star-poly-concentration} that the $W_p$ have bounded variance;
the following proposition shows that this also holds under the
concentration condition we have put on $\|X\|^2$.

\begin{prop}\label{T:variance-bound-hermitian}
Let $X$ be a random matrix in $\symmat{n}{\C}$ as above, whose
distribution is invariant under rotations in $\symmat{n}{\C}$, and let $W_p=\tr(X^p)$.
Suppose that $\E\|X\|^2=n$ and that for each $k$, there is a constant $\alpha_k$ depending only on $k$ such that 
\[t_k(X)=\left|n^{-k/2}\E\|X\|^{k}-1\right|\le\frac{\alpha_k}{n^2}.\] 
Then for each fixed $p\in\N$, there are constants $\kappa_{p,2}$ and
$\kappa_{p,4}$, depending on
$p$ and the $\alpha_k$
but not $n$, such that
\[
\var(W_p) \le \kappa_{p,2}
\qquad \text{and} \qquad 
\E (W_p - \E W_p)^4 \le \kappa_{p,4} n^2.
\]
\end{prop}

\begin{proof}
  As above, we write $X = \frac{\norm{X}}{\sqrt{n}}\widetilde{X}$,
  where $\widetilde{X}$ is uniformly distributed on the sphere of
  radius $\sqrt{n}$ in $\symmat{n}{\C}$ and $\widetilde{X}$ is
  independent from $\norm{X}$.  Let $R:=\frac{\|X\|}{\sqrt{n}}$ and
  $\widetilde{W}_p=\tr(\widetilde{X}^p)$.
  We have
  \begin{align*}
    W_p - \E W_p & = R^p \widetilde{W}_p - (\E R^p) (\E
                   \widetilde{W}_p) \\
    & = R^p (\widetilde{W}_p - \E \widetilde{W}_p) + (\E
    \widetilde{W}_p)\bigl[ (R^p - 1) - (\E R^p - 1)\bigr]
  \end{align*}
  and therefore 
  \begin{align*}
    \left(\E \abs{W_p - \E W_p}^q\right)^{1/q}
    & \le \left(1 + t_{pq}(X) \right)^{1/q} \left(\E \abs{\widetilde{W}_p -
        \E \widetilde{W}_p}^q\right)^{1/q} \\
    & \qquad + \abs{\E \widetilde{W}_p} \left[\left(\E \abs{R^p-1}^q\right)^{1/q}
      + t_p(X)\right]
  \end{align*}
  for any $q \ge 1$ by the $L^q$ triangle inequality. By Proposition
  \ref{T:star-poly-concentration}, Theorem
  \ref{T:traces-of-powers-means}, and the fact that
  $t_k(X) = O(n^{-2})$ for each $k$, we have
  $\left(\E \abs{\widetilde{W}_p - \E \widetilde{W}_p}^q\right)^{1/q}
  = O(1)$, $\abs{\E \widetilde{W}_p} = O(n)$,
  $\left(\E R^{pq} \right)^{1/q} = O(1)$, and
  $\abs{\E R^p - 1} = O(n^{-2})$.

  To complete the proof, observe that
  \[
    \E (R^p - 1)^2 = (\E R^{2p} - 1) - 2 (\E R^p - 1) \le t_{2p}(X) +
    2 t_p(X) = O(n^{-2})
  \]
  and similarly
  \[
    \E (R^p - 1)^4 \le t_{4p}(X) + 4 t_{3p}(X) + 6 t_{2p}(X) + 4
    t_p(X) = O(n^{-2}).
    \qedhere
  \]
\end{proof}

\begin{proof}[Proof of Theorem \ref{T:traces-of-powers-clt-normal} for $V=\symmat{n}{\C}$]
  We begin with condition \eqref{inf-lincond} of Theorem
  \ref{T:inf-abstract}.  We write $R:=\frac{\|X\|}{\sqrt{n}}$,
  $Y_p:=W_p-\E W_p$, and
  $Y_{\epsilon,p}:=W_{\epsilon,p}-\E W_{\epsilon,p}$. By
  \eqref{E:csym-linear1},
\begin{equation}\begin{split}\label{E:lin-diff-Hermite}
    \lim_{\epsilon\to0}\frac{1}{\epsilon^2}
    &\E\left[Y_{\epsilon,p}-Y_p\big|X\right]\\
    &=\sum_{\ell=0}^{p-2}\frac{p}{d(d-1)}\left[W_2W_\ell W_{p-2-\ell}
      -\E\left[W_2W_\ell W_{p-2-\ell}\right]\right]
    -\frac{p(p+d-2)}{d(d-1)}Y_p\\
    &=\sum_{\ell=0}^{p-2}\frac{pn}{d(d-1)}\left[R^2W_\ell W_{p-2-\ell}
      -\E\left[R^2W_{p-2-\ell}W_\ell\right]\right]
    -\frac{p(p+d-2)}{d(d-1)}Y_p.
\end{split}\end{equation}

Observe that
\begin{align*}
R^2 W_\ell
  W_{p-2-\ell} -\E\left[R^2W_{p-2-\ell}W_\ell\right] = Y_\ell\E W_{p-2-\ell}+Y_{p-2-\ell}\E
                  W_\ell+Y_2\left(\frac{\E W_\ell\E
                  W_{p-2-\ell}}{n}\right)+F_{p,\ell},
\end{align*}
where
\begin{equation}\label{Eq:Fpl}
\begin{split}
F_{p,\ell}  
&:=  R^2Y_\ell
                  Y_{p-2-\ell}-\E[R^2Y_\ell
                  Y_{p-2-\ell}]+\left[(R^2-1)Y_\ell-\E[(R^2-1)Y_\ell]\right]\E
                                   W_{p-2-\ell}\\
&\qquad+\left[(R^2-1)Y_{p-2-\ell}-\E[(R^2-1)Y_{p-2-\ell}]\right]\E
                                   W_{\ell}.\\
&=  (R^2-1)Y_\ell
                  Y_{p-2-\ell} -\E[(R^2-1) Y_\ell
                  Y_{p-2-\ell}] + Y_\ell Y_{p-2-\ell} - \E [Y_\ell
                  Y_{p-2-\ell}] \\
&\qquad +\left[(R^2-1)Y_\ell-\E[(R^2-1)Y_\ell]\right]\E
                                   W_{p-2-\ell} \\
& \qquad +\left[(R^2-1)Y_{p-2-\ell}-\E[(R^2-1)Y_{p-2-\ell}]\right]\E
                                   W_{\ell}.
\end{split}
\end{equation}

Using this expression in \eqref{E:lin-diff-Hermite} yields
\begin{equation*}\begin{split}
\lim_{\epsilon\to0}&\frac{1}{\epsilon^2}\E\left[Y_{\epsilon,p}-Y_p\big|X\right]\\
&=\sum_{\ell=0}^{p-2}\frac{pn}{d(d-1)}\left[Y_\ell\E W_{p-2-\ell}+Y_{p-2-\ell}\E
                  W_\ell+Y_2\left(\frac{\E W_\ell\E
                  W_{p-2-\ell}}{n}\right)+F_{p,\ell}\right]-\frac{p(p+d-2)}{d(d-1)}Y_p
            \\
&=\frac{p}{d(d-1)} \left[2n \sum_{\ell=0}^{p-2} Y_\ell\E
    W_{p-2-\ell}+ Y_2\sum_{\ell=0}^{p-2}\E W_\ell\E W_{p-2-\ell}
    +n \sum_{\ell=0}^{p-2} F_{p,\ell}
    - (p+d-2) Y_p\right].
\end{split}\end{equation*}

As in the statement of Theorem \ref{T:traces-of-powers-clt-normal}, take
\[
Z_p=Y_p-\frac{p\E W_p}{2n}Y_2,
\]
for $p\ge 0$ (in particular, $Z_0 = Z_2 = 0$), and
\[
Z_{\epsilon,p}=Y_{\epsilon,p}-\frac{p\E W_p}{2n}Y_{\epsilon,2}
=Y_{\epsilon,p}-\frac{p\E W_p}{2n}Y_2,
\]
where the last equality follows since
$W_{\epsilon,2} = \|X_\epsilon\|^2 = \|X\|^2 = W_2$, so that
$Z_{\epsilon,p} - Z_p = Y_{\epsilon,p} - Y_p$.  We then have
\begin{equation}\begin{split}\label{E:Z-linear}
\lim_{\epsilon\to0}\frac{1}{\epsilon^2}
\E\left[Z_{\epsilon,p}-Z_p\big|X\right]
=\frac{p}{d(d-1)} &\left[2n\sum_{\ell=0}^{p-2}Z_\ell\E W_{p-2-\ell}
+(p+1) Y_2 \sum_{\ell=0}^{p-2}\E W_\ell\E
  W_{p-2-\ell} \right.\\
&\qquad\left.+n\sum_{\ell=0}^{p-2} F_{p,\ell} - (p+d-2)\left(Z_p+\frac{p\E W_p}{2n}Y_2\right)\right].
\end{split}\end{equation}
By Theorem \ref{T:traces-of-powers-means},
\[
\sum_{\ell=0}^{p-2} \E W_\ell \E W_{p-2-\ell} = n^2
\sum_{\ell=0}^{p-2} C_{\ell} C_{p-2-\ell} = O(n^2)
\]
if $p$ is even, and is $0$ otherwise.  Equation \eqref{E:Z-linear}
therefore implies that
\begin{align*}
\lim_{\epsilon\to0}\frac{1}{\epsilon^2}&\E\left[Z_{\epsilon,p}-Z_p\big|X\right]
=\frac{p}{d(d-1)}\left[2n \sum_{\ell=0}^{p-2} Z_\ell\E
  W_{p-2-\ell} -(p+d-2) Z_p + n \sum_{\ell=0}^{p-2} F_{p,\ell} + O(n^2) Y_2\right].
\end{align*}

We define a matrix $\Lambda$ with entries indexed by
$p,q\in\{1,\ldots,m\}\setminus\{2\}$ as follows:
\[
\big[\Lambda\big]_{pq}=\begin{cases}-\frac{2pC_{(p-2-q)/2}}{d-1} & 
  \text{if $1\le
    q\le p-2$, $q\neq 2$, and $p-2-q$ is even};
  \\\frac{p}{d-1} & \text{if }q=p;\\0 &\text{otherwise}.\end{cases}
\]
Then $\Lambda=\frac{1}{d-1}T$ where $T$ is an invertible, lower
triangular matrix which is independent of $n$, and so
$\|\Lambda^{-1}\|_{op}=(d-1)\|T^{-1}\|_{op}\le \kappa_mn^2.$

We define the random error $E$  to be 
\[
E=\lim_{\epsilon\to 0}\E\left[Z_\epsilon-Z\big|X\right]+\Lambda Z,
\]
so that 
\begin{equation}\begin{split}\label{E:Hermitian-linear-error}
E_p&=\frac{2p}{d-1}\sum_{\ell=0}^{p-2}Z_\ell\left(\frac{\E
    W_{p-2-\ell}}{n}-C_{(p-2-\ell)/2}\right)
-\frac{p(p-2)}{d(d-1)}Z_p
+ \frac{2 p}{n(d-1)}\sum_{\ell=0}^{p-2}F_{p,\ell} +O(n^{-2}) Y_2.
\end{split}\end{equation}

Since $Y_\ell$ and $Y_2$ are centered with bounded variance and
$\E W_\ell = O(n)$, $Z_\ell$ is centered with bounded variance as
well. We also claim that $\E \abs{F_{p,\ell}} = O(1)$.

To see this, observe first that by H\"older's inequality and
Proposition \ref{T:variance-bound-hermitian},
\begin{align*}
  \E \abs{(R^2 - 1) Y_\ell Y_{p-2-\ell}} 
  & \le \left(\E (R^2 - 1)^2\right)^{1/2} \left(\E
    Y_\ell^4\right)^{1/4} \left(\E Y_{p-2-\ell}^4\right)^{1/4}  = O(1)
\end{align*}
where $\E (R^2 - 1)^2$ is bounded as in the proof of Proposition
\ref{T:variance-bound-hermitian}. The other terms in
\eqref{E:Hermitian-linear-error} are bounded in $L^1$ similarly. It
then follows that \( \E|E|\le \kappa_m n^{-3}.  \)

Finally, we consider part \ref{P:quad_bar} of Lemma
\ref{T:limits_all_spaces}. In the present context, the
Hilbert--Schmidt inner product is real and therefore coincides with
the real inner product $\inprod{\cdot}{\cdot}$.  Therefore
\begin{align*}
  \sum_{\alpha}\tr(X^{p-1}B_\alpha)\tr(X^{q-1}B_\alpha)
  =\sum_\alpha\inprod{X^{p-1}}{B_\alpha}\inprod{X^{q-1}}{B_\alpha}=\inprod{X^{p-1}}{X^{q-1}}=W_{p+q-2}
\end{align*}
and
\begin{align*}
\sum_\alpha X_\alpha\tr(X^{p-1}B_\alpha)
= \sum_{\alpha} \inprod{X}{B_\alpha} \inprod{X^{p-1}}{B_\alpha}
= \inprod{X}{X^{p-1}}
=W_p.
\end{align*}
It thus follows from part (\ref{P:quad_no_bar}) of Lemma \ref{T:limits_all_spaces} that
\begin{align*}\lim_{\epsilon\to0}\frac{1}{\epsilon^2}\E\left[(W_\epsilon-W)_p(W_\epsilon-W)_q\big|X\right]=\frac{2pq}{d(d-1)}\left[W_2W_{p+q-2}-W_pW_q\right].\end{align*}
As before, $\E W_\epsilon=\E W$, so that $W_\epsilon-W=Y_\epsilon-Y$, and since
$Y_{\epsilon,2}=Y_2$, it is furthermore the case that
$Y_\epsilon-Y=Z_\epsilon-Z$.  That is, for $p,q\in\{1,\ldots,m\}\setminus\{2\}$, 
\begin{align*}
\lim_{\epsilon\to0}\frac{1}{\epsilon^2}\E\left[(Z_\epsilon-Z)_p(Z_\epsilon-Z)_q\big|X\right]=\frac{2pq}{d(d-1)}\left[W_2W_{p+q-2}-W_pW_q\right].
\end{align*}

Theorem \ref{T:traces-of-powers-means}, Proposition
\ref{T:variance-bound-hermitian}, and symmetry imply that
\[
  \E W_p W_q = \begin{cases}
    n^2 C_{p/2} C_{q/2} + O(n) & \text{if $p$ and $q$ are both even,} \\
    O(1) & \text{if $p$ and $q$ are both odd,} \\
    0 & \text{if $p$ and $q$ have opposite parities}.
  \end{cases}
\]

We define the matrix $\Gamma$ indexed by $p,q\in\{1,\ldots,m\}\setminus\{2\}$ by 
\[
\Gamma_{p,q} = \frac{pq}{d-1}\begin{cases}
    C_{(p+q-2)/2} - C_{p/2} C_{q/2} & \text{if $p$ and $q$ are both even,} \\
    C_{(p+q-2)/2} & \text{if $p$ and $q$ are both odd,} \\
    0 & \text{if $p$ and $q$ have opposite parities}.
  \end{cases}
\]
and let $\Sigma=\Lambda^{-1} \Gamma$.  With these choices of $\Lambda$
and $\Sigma$, the random error matrix $E'$ of Theorem
\ref{T:inf-abstract} can be bounded as in the previous sections to
complete the proof. As noted in the introduction, it is not obvious
from this form that $\Sigma$ is positive semidefinite.  However, the
argument above shows that $\Sigma$ arises as the limit of a sequence
of covariance matrices, and therefore must be positive semidefinite.
\end{proof}

\section{Rotationally invariant ensembles in $\symmat{n}{\R}$}\label{S:symmetric}

Proceeding by Lemma \ref{T:limits_all_spaces} as before, let $A\in\symmat{n}{\R}$ and recall that in the case of $\symmat{n}{\R}$, $\{B_\alpha\}_{\alpha=1}^d=\{E_{jj}\}_{j=1}^n\cup\{F_{jk}\}_{1\le j<k\le n}.$  We have
\begin{equation*}\begin{split}\sum_{\alpha=1}^dB_\alpha AB_\alpha&=\sum_{j=1}^nE_{jj}AE_{jj}+\frac{1}{2}\sum_{1\le
    j<k\le
    n}(E_{jk}+E_{kj})A(E_{jk}+E_{kj})\\&=\frac{1}{2}\sum_{j,k=1}^n E_{jk}AE_{jk}+\frac{1}{2}\sum_{j,k=1}^nE_{jk}AE_{kj}\\&=\frac{1}{2}A+\frac{1}{2}\tr(A)I,\end{split}\end{equation*}
using the fact that $A$ is symmetric.  It thus follows from Lemma \ref{T:limits_all_spaces} that
\begin{equation}\begin{split}\label{E:rsym-linear1}
\lim_{\epsilon\to0}\frac{1}{\epsilon^2}&\E\left[W_{\epsilon,p}-W_p\big|X\right]\\
&=\sum_{\ell=0}^{p-2}\frac{(\ell+1)
    \|X\|^2}{d(d-1)}\left(W_{p-2}+W_\ell
    W_{p-2-\ell}\right)-\frac{p(p+d-2)}{d(d-1)}W_p\\
&=\frac{p}{2d(d-1)} \left[(p-1)W_2W_{p-2}+W_2\sum_{\ell=0}^{p-2}
    W_\ell W_{p-2-\ell}-2(p+d-2)W_p\right],
\end{split}\end{equation}

Compare with the corresponding expression \eqref{E:csym-linear1} in
the Hermitian case: the first term here is new, but the remaining two
terms are, to top order, $\frac{1}{2}$ times the corresponding term in
the Hermitian case (recall that in $\symmat{n}{\R}$,
$d=\frac{n(n-1)}{2}$).  The first term in \eqref{E:rsym-linear1} is of
smaller order than the remaining terms, and so the proofs of Theorems
\ref{T:traces-of-powers-means} and \ref{T:traces-of-powers-clt-normal}
are essentially the same as in the complex Hermitian case.

The proof of Proposition \ref{T:variance-bound-hermitian} carries over
verbatim to this case.

Returning to condition \eqref{inf-lincond} of Theorem
\ref{T:inf-abstract}, since the expectation of the left-hand side of
\eqref{E:rsym-linear1} is zero, if $Y_p:=W_p-\E W_p$ and
$Y_{\epsilon,p}:=W_{\epsilon,p}-\E W_{\epsilon,p}$, then
\begin{align*}
\lim_{\epsilon\to0}\frac{1}{\epsilon^2}\E\left[\left.Y_{\epsilon,p}-Y_p\right|X\right]&=\frac{p(p-1)}{2d(d-1)}(W_2W_{p-2}-\E
  W_2W_{p-2})\\&\quad+\frac{p}{2d(d-1)}\sum_{\ell=0}^{p-2}\left[W_2W_\ell W_{p-2-\ell}-\E
  W_2W_\ell W_{p-2-\ell}\right]-\frac{p(p+d-2)}{d(d-1)}Y_p.
\end{align*}

Again, the first term is new and the second two are very similar to
the Hermitian case, differing only in factors of 2 that correspond to
the change in dimension.  By recentering and applying Proposition
\ref{T:star-poly-concentration}, it is straightforward to check that
the new term is of smaller order  than the others and can be
incorporated into the constant in the final bound.  The proof of
Theorem \ref{T:traces-of-powers-clt-normal} then proceeds identically to the
Hermitian case.

\section*{Acknowledgements}

This research was partially supported by grants from the U.S.\
National Science Foundation (DMS-1612589 to E.M.)\ and the Simons
Foundation (\#315593 to M.M.). The authors thank the anonymous
referees for numerous comments that improved the exposition of this
paper.


\bibliographystyle{plain}
\bibliography{traces_of_powers}

\end{document}